\documentclass[11pt]{article}
\usepackage{amsthm,amstext, amsmath,latexsym,amsbsy,amssymb,cases}

\newtheorem{cor}{Corollary}[section]

\newtheorem{lem}[cor]{Lemma}
\newtheorem{prop}[cor]{Proposition}
\newtheorem{propo}{Proposition}
\newtheorem{theor}[propo]{Theorem}

\title{\bf Refined regularity for the blow-up set at non characteristic points for the complex semilinear wave equation}
 \author{\small Asma AZAIEZ\footnote{This author is supported by the ERC Advanced Grant no. 291214, BLOWDISOL.}  \\
\small Universit\'e de Cergy-Pontoise, \\
\small AGM, CNRS (UMR 8088), 95302, Cergy-Pontoise, France.\\}
\begin{document}
\maketitle
\begin{abstract}
In this paper, we consider a blow-up solution for the complex-valued semilinear wave equation with power non-linearity in one space dimension. We show that the set of non characteristic points $I_0$ is open and that the blow-up curve is of class $C^{1,\mu_0}$ and the phase $\theta$ is $C^{\mu_0}$ on this set. In order to prove this result, we introduce a Liouville Theorem for that equation. 
 \end{abstract}

\noindent {\bf AMS classification}: 
\section{Introduction}
\subsection{The problem and known results}
We consider the following complex-valued one-dimensional semilinear wave equation
\begin{equation} \left\{
\begin{array}{l}
\displaystyle\partial^2_{t} u = \partial^2_{x} u+|u|^{p-1}u, \\
u(0)=u_{0} \mbox{ and }  u_{t}(0) = u_{1},
\end{array}
\right . \label{waveq}
\end{equation}
where $u(t):  x\in \mathbb{R}\to u(x,t) \in \mathbb{C} , u_0 \in
H^1_{loc,u}$ and $ u_1\in L^2_{loc,u}$,with $$||
v||^2_{L^2_{loc,u}}=\displaystyle\sup\limits_{a\in \mathbb{ R}}
\int_{|x-a|<1}|v(x)|^2 dx  \mbox{ and }|| v||^2_{H^1_{loc,u}}=||
v||^2_{L^2_{loc,u}}+||  \nabla v||^2_{L^2_{loc,u}}\cdot$$
\medskip
The Cauchy problem for equation (\ref{waveq}) in the space $H^1_{loc,u}\times L^2_{loc,u}$ follows from the finite speed of propagation and the wellposedness in $H^1\times L^2$. See for instance Ginibre, Soffer and Velo \cite{MR1190421}, Ginibre and Velo \cite{MR1256167}, Lindblad and Sogge \cite{MR1335386} (for the local in time wellposedness in $H^1\times L^2$). Existence of blow-up solutions follows from ODE techniques or the energy-based blow-up criterion of \cite{l74}.
 More blow-up results can be found in Caffarelli and Friedman \cite{MR849476}, Alinhac \cite{ali95} and \cite{MR1968197}, Kichenassamy and Littman \cite{KL93a}, \cite{KL93b} Shatah and Struwe \cite{MR1674843}). 
 \medskip
 
 The real case (in one space dimension) has been understood completely, in a series of papers by Merle and Zaag \cite{MR2362418}, \cite{MR2415473}, \cite{MR2931219} and \cite{Mz12} and in C\^ote and Zaag \cite{CZ12} (see also the note \cite{Mz10}). 
Recently, the authors give an  extension to higher dimensions in \cite{MZ13} and \cite{MZ13'}, where the blow-up behavior is given, together with some stability results.
 
additional zero eigenfunction in the linearized equation around the expected profile, and also because of the coupling between the real and the imaginary parts.
For other types of nonlinearities, we mention our recent contribution with Masmoudi and Zaag in \cite{3}, where we study the semilinear wave equation with exponential nonlinearity. In particular, we give the blow-up rate with some estimates.

In \cite{<3}, the author consider the complex-valued solution of (\ref{waveq}), characterize all stationary solutions and give a trapping result. In this paper, we aim at studying the structure of the set of non characteristic points and  the regularity of the blow-up curve and the phase.

Let us first introduce some notations before stating our results.

\medskip
If \emph{u} is a blow-up solution of (\ref{waveq}), we define (see
for example Alinhac \cite{ali95}) a continuous curve $\Gamma$ as the graph of a function ${x \rightarrow T(x)}$ such that the domain of definition of $u$ (or the maximal influence domain of $u$) is 
\begin{equation}\label{domaine-de-definition}
D_u=\{(x,t)| t<T(x)\}.
\end{equation}

From the finite speed of propagation, $T$ is a 1-Lipschitz function. The time $\bar {T}=\inf_{x \in \mathbb R}T(x)$ and the graph $\Gamma$  are called (respectively) the blow-up time and the blow-up graph of $u$.

Let us introduce the following non-degeneracy condition for $\Gamma$. If we introduce for all $x \in \mathbb{R},$ $t\le T(x)$ and $\delta>0$, the cone
\begin{equation*}
 \mathcal{C}_{x,t, \delta }=\{(\xi,\tau)\neq (x,t)\,|  \tau\le t-\delta |\xi-x|\},
\end{equation*}
then our non-degeneracy condition is the following: $x_0$ is a non-characteristic point if
\begin{equation}\label{4}
 \exists \delta_0 = \delta (x_0) \in (0,1) \mbox{ and } t_0 (x_0) < T(x_0) \mbox{ such that } u \mbox{ is defined on }\mathcal{C}_{x_0,T(x_0), \delta_0}\cap \{t \ge t_0 \}.
\end{equation}
If condition (\ref{4}) is not true, then we call $x_0$ a characteristic point. Already when $u$ is real-valued, we know from \cite{MR2931219} and \cite{CZ12} that there exist blow-up solutions with characteristic points.

We denote by $I_0$ the set of non characteristic points.

\medskip
Given some $(x_0,T_0) $ such that $0<T_0 \le T(x_0)$, we introduce the following self-similar change of variables:
\begin{equation}\label{trans_auto}
w_{x_0,T_0}(y,s) =(T_0-t)^\frac{2}{p-1}u(x,t), \quad  y=\frac{x-x_0}{T_0-t}, \quad
s=-\log(T_0-t).
\end{equation}
If $T_0= T(x_0)$, then we write $w_{x_0}$ instead of $w_{x_0,T_0}$.
This change of variables transforms the backward light cone with vertex $(x_0, T(x_0))$ into the infinite cylinder $(y,s)\in (-1,1) \times [-\log T(x_0),+\infty).$ The function $w_{x_0}$ (we write $w$ for simplicity) satisfies the following equation for all $|y|<1$ and $s\ge -\log T_0$:
\begin{eqnarray} \label{equa}
 \partial^2_{s} w=\mathcal{L}w-\frac{2(p+1)}{(p-1)^2}w+|w|^{p-1}w-\frac{p+3}{p-1} \partial_s w- 2 y \partial_{ys} w
 \end{eqnarray}
 \begin{eqnarray}\mbox{where }\mathcal{L} w=\frac{1}{\rho}\partial_y (\rho (1-y^2)\partial_y w)\, \mbox{ and }\, \rho (y)= (1-y^2)^\frac{2}{p-1}.\label{8}
\end{eqnarray}
This equation will be studied in the space
 \begin{eqnarray}\label{9}
\mathcal{ H}=\{ q \in H_{loc}^1 \times L_{loc}^2 ((-1,1),\mathbb{C}) \, \Big| \parallel  q \parallel_{\mathcal{ H}}^2 \equiv \int_{-1}^{1}(|q_1|^2+|q'_1|^2(1-y^2)+|q_2|^2)\rho\;dy< +\infty \} ,
 \end{eqnarray}
which is the energy space for $w$. Note that $\mathcal{ H}=\mathcal{ H}_0\times L_{\rho}^2$ where
 \begin{eqnarray*}\label{}
\mathcal{ H}_0=\{ r \in H_{loc}^1  ((-1,1),\mathbb{C}) \,\Big| \parallel  r \parallel_{\mathcal{ H}_0}^2 \equiv \int_{-1}^{1}(|r'|^2(1-y^2)+|r|^2)\rho\;dy< +\infty\} .
 \end{eqnarray*}
Let us define 
\begin{equation}\label{Lyupunov}
 E(w,\partial_s w)=\int_{-1}^{1} \left( \frac{1}{2} |\partial_s w|^2+\frac{1}{2} |\partial_y w|^2 (1-y^2)+\frac{p+1}{(p-1)^2}|w|^2-\frac{1}{p+1}|w|^{p+1}\right) \rho dy.
\end{equation}
By the argument of Antonini and Merle \cite{AM01}, which works straightforwardly in the complex case, we see that $E$ is a Lyapunov functional for equation (\ref{equa}).
Similarly, some arguments of the real case, can be adapted with no problems to the complex case, others don't.

Let us first briefly state our main result in \cite{<3} , then we give the main results of this paper.
\bigskip

In \cite{<3}, we proved the existence of the blow-up profile at non-characteristic points. More precisely, this is our statement (see Theorem $4$ page $5895$ in \cite{<3}).\\
\medskip

{\it There exist positive $\mu_0$ and $C_0$ such that if $u$ a solution of (\ref{waveq}) with blow-up curve $\Gamma :\{ x \rightarrow T(x)\}$ and $x_0 \in \mathbb R$ is non-characteristic (in the sense (\ref{4})), then there exists $d_{\infty} \in (-1,1)$ and $\theta_\infty\in \mathbb{R}$, $s_0(x_0)\ge -\log T(x_0)$ such that for all $s\ge s^*(x_0)$:

  \begin{eqnarray}\label{10}
 \Big|\Big|\begin{pmatrix} w_{x_0}(s)\\\partial_s w_{x_0}(s) \end{pmatrix} -e^{i \theta (x_0)}\begin{pmatrix} \kappa(d(x_0),.)\\0\end{pmatrix} \Big|\Big|_{\mathcal{ H}}\le C_0 e^{-\mu_0(s-s(x_0))}.
  \end{eqnarray}

  $\kappa(d,y)$ is given by the following:
\begin{eqnarray}\label{defk}
\forall (d, y) \in (-1,1)^2,\;\kappa(d,y)= \kappa_0 \frac{(1-d^2)^\frac{1}{p-1}}{(1+dy)^\frac{2}{p-1}} \mbox{ and }\kappa_0=\left(\frac{2(p+1)}{(p-1)^2}\right)^\frac{1}{p-1},
\end{eqnarray}
 Moreover, we have
   \begin{equation*}
    E(w(s),\partial_s w(s)) \ge E(\kappa_0,0) \mbox{ as } s\rightarrow +\infty
   \end{equation*}  
     and
  \begin{equation*}
 ||w_{x_0}(s)-e^{i\theta(x_0)} \kappa(d(x_0))||_{H^1(-1,1)}+||\partial_s w_{x_0}(s)||_{L^2(-1,1)} \rightarrow  0 \mbox{ as } s\rightarrow +\infty.
  \end{equation*}
}
\medskip

In the real case, relying on the existence of a blow-up profile, together with Liouville type Theorem, Merle and Zaag could prove the openness of the set of non-characteristic points $I_0$ and the $C^1$ regularity of the blow-up curve restricted to $I_0$. Later, in \cite{N} Nouaili improved this by showing the $C^{1,\alpha}$ regularity of $T$. In this paper, we aim at showing the same result. In fact, the situation is more delicate since we have to deal with the regularity of the phase, a new further with respect to the real case. More precisely, this is our main result:
\begin{theor}\label{1}{\bf (Regularity of the blow-up set and continuity of the blow-up profile on $I_0$)}
Consider $u$ a solution of (\ref{waveq}) with blow-up curve $\Gamma :\{ x \rightarrow T(x)\}$.
Then, the set of non characteristic points $I_0$ is open and $T(x)$ is of class $C^{1, \mu_0}$ and  $\theta$ is of class $C^{ \mu_0}$ on that set. Moreover, for all $x \in I_0$, $T'(x) = d(x) \in (-1, 1)$ on connected components of $I_0$, where $d(x)$ and $\theta(x)$ are such that (\ref{10}) holds.
\end{theor}

Note that the holder parameter $\mu_0$ is the same as the parameter displayed in the exponential convergence to the profile given in (\ref{10}).

The proof of this theorem relies on this Liouville Theorem:


\begin{theor}\label{2}{\bf (A Liouville Theorem for equation (\ref{waveq}))}
Consider $u(x,t)$ a solution to equation (\ref{waveq}) defined on the cone $\mathcal{C}_{x_*,T_*,\delta_*}$ (\ref{4}) such that for all $t < T^*$,
\begin{align}\label{15}
&(T^*-t)^\frac{2}{p-1}\frac{||u(t)||_{L^2(B(x^*, \frac{T^*-t}{ \delta_*}))}}{(T^*-t)^{1/2}}\notag\\
&+(T^*-t)^{\frac{2}{p-1}+1}\left( \frac{||\partial_t u(t)||_{L^2(B(x^*,\frac{T^*-t}{ \delta_*})}}{(T^*-t)^{1/2}}+\frac{||\nabla u(t)||_{L^2(B(x^*,\frac{T^*-t}{ \delta_*}))}}{(T^*-t)^{1/2}}\right) \le  C^*,
\end{align}
fore some $(x_*, T_*)  \in  \mathbb R^2 , \, \delta_* \in (0,1)$ and $C^*>0$. \\
Then, either $u \equiv 0$ or $u$ can be extended to a function (still denoted by $u$) defined in
\begin{equation}\label{16}
\{(x,t)|    t<T_0+d_0 (x-x^*)  \}\supset \mathcal{C}_{x_*,T_*,\delta_*}  \mbox{ by } u(x,t)=  e^{i\theta_0} \kappa_0 
 \frac{(1-d_0^2)^{\frac{1}{p-1}}}{\left(  T_0-t+ d_0 (x-x^*)\right)^\frac{2}{p-1}},
\end{equation}
for some $T_0  \ge  T^*,\, d_0 \in [-\delta_*, \delta_*]$ and $\theta_0 \in   \mathbb R$, where $  \kappa_0$ defined in (\ref{defk}).
\end{theor}

Applying the self-similar variables' transformation (\ref{trans_auto}), we get this equivalent Theorem:
\begin{theor}\label{2'}{\bf (A Liouville Theorem for equation (\ref{equa}))}
Consider $w(y,s)$ a solution to equation (\ref{equa}) defined for all $(y,s)  \in (- \frac{1}{ \delta_*},  \frac{1}{ \delta_*})  \times  \mathbb R$ such that for all $s \in   \mathbb R$,
\begin{equation}\label{17}
||w_{x_0}(s)||_{H^1(- \frac{1}{ \delta_*},  \frac{1}{ \delta_*})}+||\partial_s w_{x_0}(s)||_{L^2(- \frac{1}{ \delta_*},  \frac{1}{ \delta_*})}\le C^*
\end{equation}
fore some $ \delta_* \in (0,1)$ and $C^*>0$. Then, either $w \equiv 0$ or $w$ can be extended to a function (still denoted by $w$) defined in
\begin{equation}\label{17}
\{(y,s)|  -1-T_0 e^s <d_0 y \} \supset \left(- \frac{1}{ \delta_*} ,  \frac{1}{ \delta_*}\right) \times  \mathbb R  \mbox{ by } w(y,s)=  e^{i\theta_0} \kappa_0 
 \frac{(1-d_0^2)^{\frac{1}{p-1}}}{\left( 1+ T_0 e^s+d_0 y\right)^\frac{2}{p-1}},
\end{equation}
for some $T_0  \ge  T^*,\, d_0 \in [ -\delta_*, \delta_*]$ and $\theta_0 \in   \mathbb R$, where $  \kappa_0$ defined in (\ref{defk}).
\end{theor}

This paper is organized as follows:\\
- In section 2, we give the proof of Theorem \ref{1} assuming the Liouville Theorem.\\
- In section 3, we state briefly some previous results for the complex-valued solution of (\ref{waveq}), then give the outline of the proof of the Liouville Theorem.\\
\section{Regularity of the blow-up curve}
In this section, we give the outline of the proof of Theorem \ref{1}. 
In order to do so, we proceed in 4 steps:\\
- In Step $1$, we assume Theorem \ref{2}, and we study the differentiability of the blow-up curve at a given non characteristic point.\\
- In Step $2$, we give two geometrical results for a non characteristic point.\\
- In Step $3$, we use the results of the two previous parts to show that $I_0$ is open and that $T$ is $C^1$ on this set.\\
- In Step $4$, adapting the strategy of Nouaili \cite{N}, we refine the result of Step $3$ and prove that $T$ is of class $C^{1, \mu_0}$ and  the phase $\theta$ is of class $C^{ \mu_0}$ on $I_0$.
\bigskip

\noindent $\rightsquigarrow$ {\it Step $1$: Differentiability of the blow-up curve at a given non characteristic point.}\\
In this step, we give the recall the result of the real case page $60$ in \cite{MR2415473} which remains valid in the complex case with no change. For the reader convenience we introduce the result and we give the outline of the proof.
\begin{prop}\label{2.1}{\bf (Differentiability of the blow-up curve at a given non characteristic point) } If $x_0$ is a non characteristic point, then $T(x)$ is differentiable at $x_0$ and $T'(x_0) = d(x_0)$ where $d(x_0)$ is such that (\ref{10}) holds.
\end{prop}

\begin{proof}[Proof of Proposition \ref{2.1}] From translation invariance, we can assume that $x_0 =T(x_0)=0$, we assume also that $\theta(0)=1$.
In order to prove that $T(x)$ is differentiable when $x = 0$ and that $T'(0) = d(0)$, we proceed by contradiction. From the fact that $T(x)$ is 1-Lipschitz, we assume that there is a sequence $x_n$ such that
 \begin{equation} \label{30}
x_n \rightarrow  0 \mbox{ and } T(x_n)  \rightarrow  d(0) + \lambda  \mbox{ with }  \lambda \neq 0  \mbox{ as }  n   \rightarrow \infty. 
  \end{equation}
Up to extracting a subsequence and to considering $u(-x, t)$ (also solution to (\ref{waveq})), we can assume that $x_n > 0$.\\
We recall the following
\begin{cor}\label{2.4}  Let $\delta_1 =\frac{1+\delta_0'}{2}$. For $\sigma_n'= -\log \left(\frac{\delta_1 (T(x_n)+\delta_0' x_n)}{\delta_1-\delta_0'} \right)$, we have
 \begin{eqnarray*}\label{}
 \Big|\Big|\begin{pmatrix} w_{x_n}(\sigma_n')\\\partial_s w_{x_n}(\sigma_n') \end{pmatrix} -\begin{pmatrix} w_{\pm}(\sigma^*)\\\partial_s w_{\pm}(\sigma^*) \end{pmatrix} \Big|\Big|_{H^1\times L^2(-\frac{1}{\delta_1},\frac{1}{\delta_1}) } \rightarrow 0   \mbox{ as }  n   \rightarrow \infty. 
 \end{eqnarray*}
  where $\pm=- sgn \lambda$,
   \begin{equation*}
\sigma^*=\log \left(\frac{|\lambda| (\delta_1-\delta_0)}{\delta_1(\lambda+d(0)+\delta_0')} \right) \mbox{ and } w_\pm(y,s)=  \kappa_0 
 \frac{(1-d(0)^2)^{\frac{1}{p-1}}}{\left( 1\pm e^s+d(0) y\right)^\frac{2}{p-1}},
 \end{equation*}
  is a solution to (\ref{equa}).
\end{cor}
\begin{proof} 
The proof is the same as in the real (see page $63$ in \cite{MR2415473}), one can adapte it without difficulty.
\end{proof} 
\noindent We discuss within the sign of $\lambda$:\\
\noindent{\bf Case $\lambda < 0$:} Here, we will reach a contradiction using Corollary \ref{2.4} and the fact that $u(x, t)$ cannot be extended beyond its maximal influence domain $D_u$ defined by (\ref{domaine-de-definition}).

\noindent{\bf Case $\lambda  > 0$:} Here, a contradiction follows from the fact that $w_{x_n} (y, s)$ exists for all $(y, s) \in (-1, 1) \times [- \log T (x_n), +\infty)$ and satisfies a blow-up criterion (given in Theorem $2$ page $1147$ in the paper of Antonini and Merle \cite{AM01}, which is available also for a complex-valued solution) at the same time.
\\Thus, (\ref{30}) does not hold and $T(x)$ is differentiable at $x = 0$ with $T'(0) = d(0)$. This
concludes the proof of Proposition \ref{2.1}.
\end{proof} 

\noindent $\rightsquigarrow$ {\it Step $2$: Openness of the set of $x$ such that (\ref{10}) holds}\\
We have from the dynamical study in self-similar variables (\ref{trans_auto}) we have the following
\begin{lem}\label{2.6}  {\bf (Convergence in self-similar variables for $x$ close to $0$)} For all $\epsilon > 0$, there exists $\eta$ such that if $|x| < \eta$ and $x$ is non characteristic, then, (\ref{10}) holds for $w_x$ with $|d(x)-d(0)|\le \epsilon$ and $|\theta(x)-\theta(0)|\le \epsilon$.
\end{lem}
\begin{proof} 
We proceed as in the real case in page $66$ in \cite{MR2415473}, considering our two pamareters $d$ and $\theta$ instead of one parameter $d$  in the real case.
\end{proof} 
We claim:
\begin{lem}\label{2.7} {\bf (The slope of $T (x)$ around $0$ is less than $(1 + |d(0)|)/2$)} It holds that
 \begin{equation}
  \forall x, y  \in [-\frac{\eta_0}{10},\frac{\eta_0}{10}] ,  |T(x)-T(y)|\le 1+|d(0)| |x-y|. 
  \end{equation}
\end{lem}
\begin{proof} 
 The proof in the real case stay valid without any change in the complex case. In fact, we never use the profile of $w$,we use only a geometrical constuction. For more details see page $67$ in \cite{MR2415473}.
\end{proof} 
\noindent $\rightsquigarrow$ {\it Step $3$: $C^1$ regularity of the blow-up set}\\
Let $x_0$ be a non characteristic point. 
 One can assume that $x_0 = T(x_0) = 0$ from translation invariance. From \cite{<3} and Proposition \ref{2.1}, we know (up to replacing $u(x, t)$ by $-u(x, t)$) that (\ref{10}) holds with some $d(0) \in (-1, 1)$ and $\theta(0) = 1$, and that $T(x)$ is differentiable at $0$ with
 \begin{equation} \label{49}
T'(0) =d(0). 
  \end{equation}

Using Lemma \ref{2.7}, we see that for all $x \in [-\frac{\eta_0}{20} , \frac{\eta_0}{20} ]$, $x$ is non characteristic in the sense (\ref{4}). Using Proposition \ref{2.1}, we see that $T$ is differentiable at $x$ and $T' (x) = d(x)$ where $d(x)$ is such that (\ref{10}) holds for $w_x$. Using Lemma \ref{2.6}, we see from (\ref{49}) that $T'(x) = d(x) \rightarrow d(0) = T'(0)  \mbox{ as } x  \rightarrow 0  \mbox{ and } \theta(x) = 1$.

\medskip

\noindent $\rightsquigarrow$ {\it Step $4$: $C^{1,\mu_0}$ regularity of the Blow-up curve and $C^{\mu_0}$ regularity the phase $\theta$.}\\
In this step, we conclude the proof of Theorem \ref{1}. In order to do so, we use in addition to the techniques used in the real case in \cite{N}, which remains valid in our case, a decomposition into real and imaginary parts in some inequalities, which gives a new information concerning the regularity of the phase.

We introduce the following:
\begin{lem}{\bf (Locally uniform convergence to the blow-up profile)}
 There exist positive $\mu_0=\mu_0(p)$ and $C_0=C_0(p)$ such that for all $x_0 \in \mathcal{R}$, there exist $\delta>0$, $s^*\in \mathbb{R}$, such that for all $X\in (x_0-\delta, x_0+\delta)$ and $s\ge s^*$,
  \begin{eqnarray}\label{}
 \Big|\Big|\begin{pmatrix} w_{x}(s)\\\partial_s w_{x}(s) \end{pmatrix} -e^{i \theta (x_0)}\begin{pmatrix} \kappa(T'(x),.)\\0\end{pmatrix} \Big|\Big|_{\mathcal{ H}}\le C_0 e^{-\mu_0(s-s(x_0))}.
  \end{eqnarray}
\end{lem}
\begin{proof} The same idea used in the real case can be adapted to the complex case without any difficulty. It is to use the result of Lemma \ref{2.6} to prove that the convergence in (\ref{10}) is locally uniform with respect to $x_0$. For more details see page $1544$ \cite{N}.
\end{proof}
Translated back to the variables $u(x,t)$, we get the following: 
\begin{lem}
 There exist positive $\mu_0=\mu_0(p)$ and $C_0=C_0(p)$ such that for all $x_0 \in \mathcal{R}$, there exist $\delta>0$, $0<t^*(x_0)<\inf_{|x-x_0|\le \delta}T(X)$, such that for all $X\in (x_0-\delta, x_0+\delta)$ and $t\in [t^*,T(X))$,                                                                                                                                                                                                                               \begin{eqnarray}\label{155}
\underset{|\xi-X|\le \frac{3}{4}(T(X)-t)}{\sup} \Big| u(\xi,t)-   e^{i \theta (x_0)} \kappa_0\frac{(1-T'(X)^2)^\frac{1}{p-1}}{(T(X)-1+ T'(X)(\xi-X))^\frac{2}{p-1}} \big|\le C(T(X)-t)^{\mu_0-\frac{2}{p-1}}.
  \end{eqnarray}
\end{lem}

Let $x_0$ in $\mathcal{R},$ and consider an arbitrary $\sigma\ge \frac{3}{4}$. For $\delta>0$, $x\in (x_0-\sigma,x_0+\sigma)$, we define $t=t(x,\sigma)$ by:
\begin{eqnarray}\label{o}
 \frac{|x_0-x|}{T(x_0)-t}=\frac{1}{\sigma}.
\end{eqnarray}
On the one hand, using (\ref{155}) with $X=\xi=x$, we get
\begin{eqnarray}\label{op}
 \big|u(x,t)-e^{i \theta (x)} \kappa_0\frac{(1-T'(x)^2)^\frac{1}{p-1}}{(T(x)-t)^\frac{2}{p-1}} \big|\le C(T(x)-t)^{\mu_0-\frac{2}{p-1}}
\end{eqnarray}

On the other hand, using (\ref{155}) $X=x_0$ and $\xi=x$, we get
\begin{eqnarray}\label{opp}
 \big|u(x,t)-e^{i \theta (x_0)} \kappa_0\frac{(1-T'(x_0)^2)^\frac{1}{p-1}}{(T(x_0)-1+ T'(x_0)(x-x_0))^\frac{2}{p-1}} \big|\le C(T(x)-t)^{\mu_0-\frac{2}{p-1}}
\end{eqnarray}
From (\ref{op}), (\ref{opp}) and  (\ref{o}) we derive
\begin{eqnarray}\label{oppa}
  \big|    e^{i(\theta (x_0)- \theta (x))}\frac{(1-T'(x_0)^2)^\frac{1}{p-1}}{(T'(x_0)sign (x-x_0)-1+ \sigma)^\frac{2}{p-1}} - \frac{(1-T'(x)^2)^\frac{1}{p-1}}{(\frac{T(x)-T(x_0)}{|x-x_0|}+\sigma)^\frac{2}{p-1}}  \big|\le C |x_0-x|^{\mu_0},
\end{eqnarray}
where $sign (x)=\frac{x}{|x|}$, for $x\neq 0$.\\

We separate the real and imaginary part in (\ref{oppa}),
\begin{eqnarray}\label{oppa1}
  \big|  \sin (\theta (x_0)- \theta (x))\frac{(1-T'(x_0)^2)^\frac{1}{p-1}}{(T'(x_0)sign (x-x_0)-1+ \sigma)^\frac{2}{p-1}}  \big|\le C |x_0-x|^{\mu_0},
\end{eqnarray}
and,
\begin{eqnarray}\label{oppa2}
  \big| \cos (\theta (x_0)- \theta (x))\frac{(1-T'(x_0)^2)^\frac{1}{p-1}}{(T'(x_0)sign (x-x_0)-1+ \sigma)^\frac{2}{p-1}} - \frac{(1-T'(x)^2)^\frac{1}{p-1}}{(\frac{T(x)-T(x_0)}{|x-x_0|}+\sigma)^\frac{2}{p-1}}  \big|\le C |x_0-x|^{\mu_0}.
\end{eqnarray}
From (\ref{oppa1}), $\big|  \sin (\theta (x_0)- \theta (x))\big|\le C |x_0-x|^{\mu_0}$. Hence, for $x$ close enough to $x_0$, we get
\begin{eqnarray*}
 \big|\theta (x_0)- \theta (x)\big|\le C |x_0-x|^{\mu_0}. 
\end{eqnarray*}
Thus, $\theta$ is $C^{1,\mu_0}$ near $x_0$.\\
In addition, for $x$ close enough to $x_0$,
\begin{eqnarray}\label{oppa3}
\cos (\theta (x_0)- \theta (x))=1+O(\theta (x_0)- \theta (x))=1+O(|x-x_0|^{2\mu_0}).
\end{eqnarray}
Using (\ref{oppa3}) with (\ref{oppa2}),
\begin{eqnarray}
  \big| \frac{(1-T'(x_0)^2)^\frac{1}{p-1}}{(T'(x_0)sign (x-x_0)-1+ \sigma)^\frac{2}{p-1}} - \frac{(1-T'(x)^2)^\frac{1}{p-1}}{(\frac{T(x)-T(x_0)}{|x-x_0|}+\sigma)^\frac{2}{p-1}}  \big|\le C |x_0-x|^{\mu_0}.
\end{eqnarray}
At this level, we reduce to the real case to conclude. We introduce a change of variables
$$f(\xi)=T(\xi+x_0)-T(x_0)-\xi T'(x_0)$$
and prove that $|f'(\xi)|\le C |\xi|^{\mu_0}$, which is equivalent to the fact that $T$ is $C^{1,\mu_0}$.

\section{Proof of the Liouville Theorem}

\subsection{Preliminaries}

In the following, we recall some results from \cite{<3}, which we have used in this work. In the following Proposition we recall some dispersion estimates.
\begin{prop}\label{3.1}{\bf (A Lyapunov functional for equation (\ref{equa}))}
Consider $w(y,s)$ a solution to (\ref{equa}) defined for all $(y,s)\in (-1,1)\times [s_0, +\infty)$ for some $s_0 \in   \mathbb R$. Then:\\
(i) For all $s_2 \ge s_1 \ge s_0$, we have
\begin{equation*}
E(w(s_2))-E(w(s_1))=-\frac{4}{p-1}\int_{s_1}^{s_2} \int_{-1}^1(\partial_s w(y,s))^2\frac{\rho(y)}{1-y^2}dy ds
\end{equation*}
where $E$ is defined in (\ref{Lyupunov}).\\
(ii) For all $s  \ge s_0+1$, $ \int_{-\frac{1}{2}}^{\frac{1}{2}} |w|^{p+1} dy\le C(E(w(s_0)+1)^p$.
\end{prop}
\begin{proof}
The proof is the same as in the real case. See \cite{AM01} for $(i)$. For $(ii)$, see Proposition $2.2$ in \cite{MR2147056} for a statement and the proof of Proposition $3.1$ page $1156$ in \cite{MZ05} for the proof.
\end{proof}
We recall the set of all stationary solutions in $\mathcal{ H}_0$ of equation (\ref{equa}).
\begin{prop}\label{p1}{\bf (Characterization of all stationary solution of equation (\ref{equa}) in  $\mathcal{ H}_0$).}
Consider $w \in \mathcal{H}_0$ a stationary solution of (\ref{equa}). Then, either $w\equiv 0$ or there exist $d\in(-1,1)$ and $\theta \in \mathbb{R}$ such that $w(y)=e^{i\theta} \kappa (d,y)$ where $\kappa(d,y)$ is given in (\ref{defk}).

\end{prop}

\begin{proof}

The proof of this Proposition present more difficulties than the real case. In fact, in addition to the techniques used in the real case, we have used an ODE techniques for complex-valued equation, in particular, a decomposition $w(y)=\rho(y)e^{i\theta (y)}$ with a delicate phase behavior $\theta (y)$. For more details see Section 2 in \cite{<3}. \\

\end{proof}

\subsection{Proofs of Theorem \ref{2'} and Theorem \ref{2} }

\begin{proof} [Proof of Theorem \ref{2'} assuming Theorem \ref{2} ] The proof is the same as in the real case. For the reader's convenience we recall it.
Consider $w(y, s)$ a solution to equation (\ref{equa}) defined for all $(y, s) \in (- \frac{1}{ \delta_*},  \frac{1}{ \delta_*})\times \mathbb R$ for some $\delta_*  \in (0,1)$ such that for all $s   \in \mathbb R$, (\ref{17}) holds. 

If we introduce the function $u(x, t)$ defined by
\begin{equation}\label{63}
u(x, t) = (-t)^\frac{-2}{ p-1} w(y, s) \mbox{ where } y =\frac{x}{-t}\mbox{ and } s =- \log(-t),
 \end{equation}
then we see that $u(x, t)$ satisfies the hypotheses of Theorem \ref{2} with $T_* = x_* = 0$, in
particular (\ref{15}) holds. Therefore, the conclusion of Theorem \ref{2} holds for $u$. Using back
(\ref{63}), we directly get the conclusion of Theorem \ref{2'}.
\end{proof}

Now, we introduce the proof of the Theorem \ref{2}.
\begin{proof} [Proof of Theorem \ref{2} ]

Consider a solution $u(x, t)$ to equation (\ref{waveq}) defined in the backward cone $\mathcal{C}_{x_*,T_*,\delta_*}$ (see (\ref{4})) such that (\ref{15}) holds, for some $(x_*,T_*) \in \mathbb R^2$ and $\delta_* \in  (0,1)$. From the bound (\ref{15}) and the resolution of the Cauchy problem of equation  (\ref{waveq}), we can extend the solution by a function still denoted by $u(x, t)$ and defined in some influence domain $D_u$ of the form
\begin{equation}\label{64}
D_u=\{(x,t)  \in  \mathbb R^2 | t<T(x)\}.
\end{equation}

for some 1-Lipschitz function $T(x)$ where one of the following cases occurs:\\
- Case 1: For all $x\in  \mathbb R, \, T(x) \equiv \infty$.\\
- Case 2: For all $x \in \mathbb R, \, T(x) < +\infty$. In this case, since $u(x,t)$ is known to be defined on $\mathcal{C}_{x_*,T_*,\delta_*}$ (\ref{4}), we have $\mathcal{C}_{x_*,T_*,\delta_*}  \subset D_u$, hence from (\ref{4}) and (\ref{64})
\begin{equation}\label{65}
\forall x  \in  \mathbb R,\; T(x)\ge T_*-\delta_*  | x-x_*|.
\end{equation}
 In this case,  we will denote the set of non characteristic points by $I_0$.

\medskip

We will treat separately these two cases: 

{\bf  Case 1: $T(x) \equiv \infty$.}\\
In the following, we give the behavior of $w_{\bar x, \bar T}(s)$ as $s\rightarrow -\infty$.

\begin{prop}\label{Behavior1}{\bf (Behavior of $w_{\bar x, \bar T}(s)$ as $s\rightarrow -\infty$)}
For any $(   \bar x,   \bar T)  \in  \bar D_u$, it holds that as $s\rightarrow -\infty$,

either
$$||w_{\bar x, \bar T}(s)||_{H^1(-1,1)}+||\partial_s w_{\bar x, \bar T}(s)||_{L^2(- 1,1)}\rightarrow 0 \mbox{ in } H^1 \times L^2(-1,1),$$
or for some $ \theta ( \bar x,  \bar T)  \in \mathbb R$
$$\inf_{\{\theta \in \mathbb R,\; |d|<1\}}||w_{\bar x, \bar T}(.,s)-e^{i\theta(\bar x, \bar T)} \kappa(d,.)||_{H^1(-1,1)}+||\partial_s w_{\bar x, \bar T}||_{L^2(-1,1)} \rightarrow  0$$
where $  \kappa_0$ defined in (\ref{defk}).
\end{prop}
Now, we derive the behavior of the Lyapunov functional $E(w_{\bar x, \bar T}(s))$ defined by (\ref{Lyupunov}) as $s\rightarrow -\infty$.

\begin{cor}\label{Behavior2}{\bf (Behavior of $E(w_{\bar x, \bar T}(s))$ as $s\rightarrow -\infty$)}
(i) For all $d  \in (-1,1)$ and $\theta \in \mathbb R,$
\begin{equation}\label{14}
 E(e^{i\theta}\kappa (d,.),0)=E(\kappa_0,0)>0
\end{equation}
(ii) For any $(\bar x, \bar T) \in \bar D_u$, either $ E(w_{\bar x, \bar T}(s)) \rightarrow 0$ or  $ E(w_{\bar x, \bar T}(s)) \rightarrow  E(\kappa_0)>0$ as $s \rightarrow  -\infty$. \\
In particular,
\begin{equation}\label{69}
\forall s \in \mathbb R, E(w_{\bar x, \bar T}(s)) \le E(\kappa_0).
\end{equation}
\end{cor}
\begin{proof} [Proofs of Proposition \ref{Behavior1} and Corollary \ref{Behavior2} ]
The proof is similar to the real case in \cite{MR2415473}, we have only to adapt it with respect to our set  of stationary solutions
\begin{equation*}
 S\equiv \{0, e^{i\theta} \kappa (d,.), |d|<1, \theta \in \mathbb{R}\}.
\end{equation*}

\end{proof}
In the following, we conclude the proof of Theorem \ref{2}, when case $1$ holds.
\begin{cor}\label{3.4}
If for all $x  \in \mathbb R$, $T(x) \equiv +  \infty$, then $u  \equiv 0$.
\end{cor}

\begin{proof} 
 In this case, $u(x, t)$ is defined for all $(x, t)  \in \mathbb{R}^2$. The conclusion is a consequence of the uniform bounds stated in the hypothesis of Theorem \ref{2} and the bound for solutions
of equation (\ref{equa}) in terms of the Lyapunov functional stated in $(ii)$ of Lemma \ref{3.1}. Indeed,
consider for arbitrary $t  \in \mathbb{R}$ and $T > t$ the function $w_{0,T}$ defined from $u(x, t)$ by means of
the transformation (\ref{trans_auto}). Note that $w_{0,T}$ is defined for all $(y, s)   \in \mathbb{R}^2$. If $s = -\log(T - t)$,
then we see from $(ii)$ in Lemma \ref{3.1} and (\ref{69}) that

$$ \int_{-\frac{1}{2}}^{\frac{1}{2}} |w_{0,T }(y,s)|^{p+1} dy\le C(E(w_{0,T }(s_0))+1)^p    \le C(E(\kappa_0)+1)^p  \equiv C1.$$
Using (\ref{trans_auto}), this gives in the original variables

$$ \int_{-\frac{T-t}{2}}^{\frac{T-t}{2}} |u(x,t)|^{p+1} dx\le C_1 (T-t)^{-\frac{2(p+1)}{p-1}+1}.$$

Fix $t$ and let $T$ go to infinity to get $u(x, t) = 0$ for all $x\in \mathbb{R}$, and then $u  \equiv 0$, which
concludes the proof of Corollary \ref{3.4} and thus the proof of Theorem \ref{2} in the case where $T(x)\equiv +\infty$.

\end{proof}
\medskip

{\bf  Case 2: $T(x) < +\infty$}\\

In this case also, we conclude by the same way as in the real case in \cite{MR2415473}. For the reader's convenience we give the three important ideas used in order to conclude the proof:\\
- In Step $1$, we localize a non characteristic point for some slop $\delta_1$.\\
- In Step $2$, we give an explicit expression of $w$ at non characteristic points.\\
- In Step $3$, we see that the set of non characteristic points is given by the hole space $\mathbb{R}$.\\
- In Step $4$, we use this three previous steps to conclude the proof when $T(x) < +\infty$.\\
\bigskip

\noindent $\rightsquigarrow$ {\it Step $1$: Localization of a non characteristic point in a given cone with slope $\delta_1>1$}: We claim the following:

\begin{prop}\label{3.5}{\bf (Existence of a non characteristic point with a given location)}

For all $x_1 \in \mathbb{R}$ and $\delta_1 \in (\delta_*, 1)$, there exists $x_0 = x_0 (x_1 ,\delta_1 )$ such that
\begin{equation}
 (x_0 , T (x_0))\in \mathcal{\bar{C}}_{x_1,T(x_1),\delta_1} \mbox{ and } \mathcal{\bar{C}}_{x_0,T(x_0),\delta_1} \subset D_u .
\end{equation}
In particular, $x_0$ is non characteristic.
\end{prop}
\begin{proof} 
In the proof we use a geometrical construction (see page 73 in \cite{MR2415473}).

\end{proof} 
\noindent {\bf Remark}: From this Proposition, we see that we have at least a non  characteristic point: In fact, Taking $x_1=x_*$ and $\delta_1=\frac{1+\delta_*}{2}$,  $x_0 \in \mathbb{R}$ is non characteristic point  (in the sense (\ref{4})).

\bigskip

\noindent $\rightsquigarrow ${\it Step $2$:  An explicit expression of $w$ at non characteristic points }: \\
We claim the following
\begin{prop}\label{3.7}{\bf (Characterization of $w_{x_0}$ when $x_0$ is non characteristic) }
If $x_0$ is non characteristic, then, there exist $d(x_0 ) \in (−1, 1)$ and $\theta(x_0) \in \mathbb{R}$ such that for all
$(y,s)\in(−1, 1)\times\mathbb{R}$, $w_{x_0} (y, s) = e^{iθ(x_0 )}\kappa(d(x_0 ), y)$.
\end{prop}
\begin{cor}\label{3.8}
 Consider $x_1 < x_2$ two non characteristic points. Then, there exists $d_0 \in
(−1, 1)$ and $\theta_0 \in \mathbb{R}$ such that:\\
$(i)$ for all $(y, s) \in(−1, 1)\times\mathbb{R}$, $w_{x_1} (y, s)= w_{x_2} (y, s)= e^{iθ_0 }\kappa(d_0 , y)$,\\
$(ii)$ for all $\bar x \in [x_1 , x_2 ], T (\bar x) = T (x_1 ) + d_0 (\bar x-x_1 )$ and for all $(x, t) \in \mathcal{C}_{\bar x,T(\bar x),1}$,
\begin{equation}
 u(x,t)=  e^{i\theta_0} \kappa_0 
 \frac{(1-d_0^2)^{\frac{1}{p-1}}}{\left( T(\bar x)-t+ d_0 (x-(\bar x))\right)^\frac{2}{p-1}}.
\end{equation}
\end{cor}

From $(ii)$ of Corollary \ref{3.8} and the remark after Proposition \ref{3.5}, we get the following:
\begin{cor}\label{3.9}
 The set of non characteristic points is a non empty interval $I_0$ and there
exist $d_0 \in (−1, 1)$ and $\theta_0$ such that on $I_0$ , the blow-up curve is a straight line with slope
$d_0$ . Moreover, for any $x_1 \in I_0$,

\begin{equation}\label{85}
 \forall (x, t) \in \cup_{\bar x \in I_0} \, \mathcal{C}_{\bar x,T(\bar x),1}, \, u(x,t)=  e^{i\theta_0} \kappa_0 
 \frac{(1-d_0^2)^{\frac{1}{p-1}}}{\left( T( x_1)-t+ d_0 (x-(x_1))\right)^\frac{2}{p-1}}.
\end{equation}
\end{cor}
\bigskip

\noindent $\rightsquigarrow ${\it Step $3$: The set of non characteristic points is given by the hole space $\mathbb{R}$}: \\

\begin{lem}
We have
 \begin{equation}
  I_0 = \mathbb{R}.
 \end{equation}
\end{lem}

\begin{proof} 
The proof is the same as in the real case in \cite{MR2415473}, we have only to replace $ \pm \theta$ by  $e^{i \theta}$. In fact, it is base on a geometrical approach, which remains valid, regardless of the real or complex value of the solution.
\end{proof}

\bigskip

\noindent $\rightsquigarrow ${\it Step $4$: Conclusion}: \\
In order to conclude the proof of Theorem \ref{2}.
If $I_0 = \mathbb{R}$, then we see from Corollary \ref{3.9} that the blow-up curve is a straight line of slope $d_0$ whose equation is
 \begin{equation}
t=T(x) \mbox{ with }\forall x \in \mathbb{R}, T(x)=T(x_*)+d_0(x-x_*) \label{87}
 \end{equation}
and that
 $$D_u =\{(x,t)|\; t<T(x_?)+d_0(x-x_*)\}$$
which contains $ \mathcal{C}_{x_*,T(x_*),1}  $ by the fact that $T(x_*) \ge T_*$ (see (\ref{65})). Using (\ref{65}) and (\ref{87}), we see that $|d_0| \le \delta_*$, hence  $\mathcal{C}_{x_*,T(x_*),\delta_*} \subset D_u$. Moreover, since $\bigcup_{\bar x \in I_0} \mathcal{C}_{\bar x,T(\bar x),1} = D_u$, we see that  (\ref{85}) implies  (\ref{16}) with $T_0 = T(x_*)$ and $x_0 = x_*$. This concludes the proof of Theorem \ref{2}.

\end{proof}

\noindent{\bf Address:}\\
Universit\'e de Cergy-Pontoise, 
Laboratoire Analyse G\'eometrie Mod\'elisation, \\
  CNRS-UMR 8088, 2 avenue Adolphe Chauvin
95302, Cergy-Pontoise, France.
\\ \texttt{e-mail: asma.azaiez@u-cergy.fr}\\
\end{document}